\documentclass[11pt,bezier]{article}
\usepackage{amsmath,amssymb,amsfonts,enumerate,setspace,graphicx,tkz-berge}

\newtheorem{prethm}{{\bf  Theorem}}

\newenvironment{thm}{\begin{prethm}{\hspace{-0.5
               em}{\bf .}}}{\end{prethm}}

\newtheorem{prepro}{{\bf  Theorem}}

\newenvironment{pro}{\begin{prepro}{\hspace{-0.5
               em}{\bf .}}}{\end{prepro}}
\newtheorem{precor}{{\bf  Corollary}}

\newenvironment{cor}{\begin{precor}{\hspace{-0.5
               em}{\bf .}}}{\end{precor}}
\newtheorem{preconj}{{\bf  Conjecture}}

\newtheorem{preremark}{{\bf  Remark}}

\newtheorem{prelem}{{\bf  Lemma}}

\newtheorem{preproof}{{\bf  Proof.}}

\newenvironment{proof}[1]{\begin{preproof}{\rm
               #1}\hfill{$\Box$}}{\end{preproof}}

\title{\large \ {\bf Planar, outerplanar and ring graph \\of the intersection graph}
\thanks
{{\it Key words}:  intersection graph, ring graph, planarity, outerplanarity. \newline
{\indent ~~{2010{ \it Mathematics Subject Classification}: 05C10, 05C25, 13A15.}}\newline
\newline}}

\author{\bf\small\sc  S. Khojasteh \\
\\ {\footnotesize {\em Department of Mathematics,  Lahijan Branch, Islamic Azad University, Lahijan, Iran}} \\ {\footnotesize {\em\newline s\_khojasteh@liau.ac.ir}}}

\date{}

\begin{document}
\maketitle
\begin{abstract}
Let $m, n > 1$ be two integers, and $\mathbb{Z}_n$ be a $\mathbb{Z}_m$-module. Let $I(\mathbb{Z}_m)^*$ be the set of all non-
zero proper ideals of $\mathbb{Z}_m$. The $\mathbb{Z}_n$-intersection graph of $\mathbb{Z}_m$, denoted by $G_n(\mathbb{Z}_m)$ is a graph with the vertex set $I(\mathbb{Z}_m)^*$, and two distinct vertices $I$ and $J$ are adjacent if and only if $I\mathbb{Z}_n\cap J\mathbb{Z}_n\neq 0$. In this paper, we determine the values of $m$ and $n$ for which $G_n(\mathbb{Z}_m)$ is planar, outerplanar or ring graph.
\end{abstract}

\section{ Introduction}
Let $R$ be a commutative ring, and $I(R)^*$ be the set of all non-zero proper ideals of $R$.
There are many papers on assigning a graph to a ring $R$, for instance see \cite{unicyclic},
\cite{conjecture}, \cite{and2}, \cite{atani} and \cite{principal}.
Also the intersection graphs of some algebraic structures such as groups, rings and modules have been studied by several authors, see \cite{akbnik, akbtaval, Chakrabarty, Cs}.
 In \cite{Chakrabarty}, the intersection graph of ideals of $R$, denoted by $G(R)$, was introduced as the graph with vertices $I(R)^*$ and for distinct $I,J\in I(R)^*$,
 the vertices $I$ and $J$ are adjacent if and only if $I\cap J\neq 0$. Also in \cite{akbtaval}, the intersection
graph of submodules of an $R$-module $M$, denoted by $G(M)$, is defined to be the graph whose vertices are
the non-zero proper submodules of $M$ and two distinct vertices are
adjacent if and only if they have non-zero intersection.
In \cite{heyGMR}, the $M$ intersection graph of ideals of $R$ denoted by $G_M(R)$, is defined to be the graph with the vertex set $I(R)^*$, and two distinct vertices $I$ and $J$ are adjacent
if and only if $IM \cap JM \neq 0$. Also, the $\mathbb{Z}_n$ intersection graph of $\mathbb{Z}_m$, $G_n(\mathbb{Z}_m)$ was studied in \cite{heyGMz}, where $n,m>1$ are integers and $\mathbb{Z}_n$ is a $\mathbb{Z}_m$-module. Clearly, if $n=m$, then $G_n(\mathbb{Z}_m)$ is exactly the same as the intersection graph of ideals of $\mathbb{Z}_m$. This implies that $G_n(\mathbb{Z}_m)$ is a generalization of $G(\mathbb{Z}_m)$. As usual, $\mathbb{Z}_m$ denotes the integers modulo $m$.

\indent Now, we recall some definitions and notations on graphs. Let $G$ be a graph with the vertex set $V(G)$ and the edge set $E(G)$. If $\{a,b\}\in E(G)$, we say $a$ is adjacent
to $b$ and write $a$ --- $b$. If $|V(G)|\geq2$, then a path from $a$ to $b$ is a series of adjacent vertices $a$ --- $x_{1}$ --- $x_{2}$ --- $\cdots$ --- $x_{n}$ --- $b$. For $a,b\in V(G)$ with $a\neq b$, $d(a,b)$ denotes the length of a shortest path from $a$ to $b$. If there is no such path, then we define $d(a,b)=\infty$.  We say that $G$ is {\it connected} if there is a path between
any two distinct vertices of $G$. A {\it cycle} is a path that begins and ends at the same vertex in which no edge is repeated and all vertices other than the starting and ending vertex are distinct. We denote the complete graph of order $n$ by $K_{n}$. A graph is {\it bipartite} if its vertices can be partitioned into two disjoint subsets $A$ and $B$ such that each edge connects a vertex from $A$ to one from $B$. A bipartite graph is a {\it complete bipartite} graph if every vertex in $A$ is adjacent to every vertex in $B$. We denote the complete bipartite graph, with part sizes $m$ and $n$ by $K_{m,n}$. The disjoint union of graphs $G_{1}$ and $G_{2}$, which is denoted by $G_{1}\cup G_{2}$, where $G_{1}$ and $G_{2}$ are two vertex-disjoint graphs, is a graph with $V(G_{1}\cup G_{2})=V(G_{1})\cup V(G_{2})$ and $E(G_{1}\cup G_{2})=E(G_{1})\cup E(G_{2})$. A graph $G$ may be expressed uniquely as a disjoint union of connected graphs. These graphs are called the connected components, or simply the components, of $G$. Recall that a graph is said to be {\it planar} if it can be drawn in the plane so that its edges intersect only at their ends. A {\it subdivision} of a graph is any graph that
can be obtained from the original graph by replacing edges by paths.

\indent In \cite{heyGMz}, the authors were mainly interested in the study of $\mathbb{Z}_n$
intersection graph of ideals of $\mathbb{Z}_m$. For instance, they determined the values of $m$ for which $G_n(\mathbb{Z}_m)$ is connected, complete, Eulerian or has a cycle. In this article, we determine all integer numbers $n$ and $m$ for which $G_n(\mathbb{Z}_m)$ is planar, outerplanar or ring graph.

\section{Results}
Let $n,m>1$ be integers and $\mathbb{Z}_n$ be a $\mathbb{Z}_m$-module. Clearly, $\mathbb{Z}_n$ is a $\mathbb{Z}_m$-module if and only if $n$ divides $m$. Throughout the paper, without loss of generality, we assume that $m=p_1^{\alpha_1}\cdots p_s^{\alpha_s}$ and $n=p_1^{\beta_1}\cdots p_s^{\beta_s}$, where $p_i$'s are distinct primes, $\alpha_i$'s are positive integers, $\beta_i$'s are non-negative integers, and $0\leq \beta_i\leq \alpha_i$ for $i=1,\ldots,s$. Let $S=\{1,\ldots,s\}$, $S'=\{i\in S \ :\,\beta _i\neq 0\}$. The cardinality of $S'$ is denoted by $s'$. Also, we denote the least common multiple of integers $a$ and $b$ by $[a,b]$. We write $a|b$ ($a\nmid b$) if $a$ divides $b$ ($a$ does not divide $b$). It is easy to see that $I(\mathbb{Z}_m)=\{d\mathbb{Z}_m : d$ divides $m \}$ and
$|I(\mathbb{Z}_m)^{*}|=\prod_{i=1}^{s}(\alpha_i+1)-2$. If $n|d$, then $d\mathbb{Z}_m$ is an isolated vertex of $G_n(\mathbb{Z}_m)$.
Obviously, $d_1\mathbb{Z}_m$ and $d_2\mathbb{Z}_m$ are adjacent in $G_n(\mathbb{Z}_m)$ if and only if $n\nmid [d_1,d_2]$.

In this section, we study the planarity of $G_n(\mathbb{Z}_m)$. A remarkable simple characterization of the planar graphs was given by Kuratowski in 1930.
\begin{pro}\label{Kuratowski}
{\rm \cite[Theorem 10.30]{bondy} (Kuratowski's Theorem) A graph is planar if and only if it contains no subdivision of either $K_5$ or $K_{3,3}$.}
\end{pro}

\begin{thm}\label{planar}{ Let $\mathbb{Z}_n$ be a $\mathbb{Z}_m$-module. Then $G_n(\mathbb{Z}_m)$ is planar if and only if one of the following holds:
\par $(1)$ $m=p_1^{\alpha_1},n=p_1^{\beta_1}$ and $\beta_1\leq 5$.
\par $(2)$ $m=p_1^{\alpha_1}p_2^{\alpha_2},n=p_1$ and $\alpha_2\leq 4$.
\par $(3)$ $m=p_1^{\alpha_1}p_2,n=p_1^{2}$.
\par $(4)$ $m=p_1^{\alpha_1}p_2p_3,n=p_1$.
\par $(5)$ $m=p_1p_2^{\alpha_2},n=p_1p_2^{2}$ and $\alpha_2=2,3,4$.
\par $(6)$ $m=p_1p_2p_3,n=p_1p_2$.
\par $(7)$ $m=p_1^{\alpha_1}p_2^{\alpha_2},n=p_1p_2$ and $\alpha_1,\alpha_2\leq 4$.
\par $(8)$ $m=n=p_1p_2p_3$.}
\end{thm}
\begin{proof}
{ One side is obvious. For the other side assume that $G_n(\mathbb{Z}_m)$ is planar. With no loss of generality assume that $n=p_1^{\beta_1}\ldots p_{s'}^{\beta_s'}$, where $1\leq s'\leq s$ and $\beta_i\neq 0$, for $i=1,\ldots,s'$. We note that $s'\leq 3$. Since otherwise, $p_1\mathbb{Z}_m,p_2\mathbb{Z}_m,p_3\mathbb{Z}_m,p_4\mathbb{Z}_m,p_1p_2\mathbb{Z}_m$ forms a $K_5$, a contradiction. Consider three following cases:

\par{\bf Case 1.} $s'=1$. If $\beta_{1}\geq 6$, then $p_1\mathbb{Z}_m,p_1^{2}\mathbb{Z}_m,p_1^{3}\mathbb{Z}_m,p_1^{4}\mathbb{Z}_m,p_1^{5}\mathbb{Z}_m$ forms a $K_5$, a contradiction. Hence $\beta_{1}\leq 5$. If $s\geq 4$, then $p_2\mathbb{Z}_m,p_3\mathbb{Z}_m,p_4\mathbb{Z}_m,p_2p_3\mathbb{Z}_m,p_2p_4\mathbb{Z}_m$ forms a $K_5$, a contradiction. Therefore $s\leq 3$. There are three following subcases:

\par{\bf Subcase 1.} $s=1$. Clearly, $G_{p_1^{\beta_1}}(\mathbb{Z}_{p_1^{\alpha_1}})$ is planar, where $\beta_1\leq 5$. Therefore $(1)$ holds.

\par{\bf Subcase 2.} $s=2$. We note that $\alpha_2\leq 4$. Otherwise, $p_2\mathbb{Z}_m,p_2^{2}\mathbb{Z}_m,p_2^{3}\mathbb{Z}_m,p_2^{4}\mathbb{Z}_m,p_2^{5}\mathbb{Z}_m$ forms a $K_5$, a contradiction. If $\beta_{1}\geq 3$, then $p_1\mathbb{Z}_m,p_1^{2}\mathbb{Z}_m,p_2\mathbb{Z}_m,p_1p_2\mathbb{Z}_m,p_1^{2}p_2\mathbb{Z}_m$ forms a $K_5$, a contradiction. This implies that $\beta_{1}=1,2$. If $\beta_{1}=1$, then it is easy to check that $G_{p_1}(\mathbb{Z}_{p_1^{\alpha_1}p_2^{\alpha_2}})$ is planar, where $\alpha_2\leq 4$. Therefore $(2)$ holds. Now, let $\beta_{1}=2$. If $\alpha_2\geq 2$, then $p_1\mathbb{Z}_m,p_2\mathbb{Z}_m,p_1p_2\mathbb{Z}_m,p_2^{2}\mathbb{Z}_m,p_1p_2^{2}\mathbb{Z}_m$ forms a $K_5$, a contradiction. Therefore $\alpha_2=1$. Obviously, $G_{p_1^{2}}(\mathbb{Z}_{p_1^{\alpha_1}p_2})$ is planar and $(3)$ holds.

\par{\bf Subcase 3.} $s=3$. If $\alpha_2\geq 2$, then $p_2\mathbb{Z}_m,p_3\mathbb{Z}_m,p_2^{2}\mathbb{Z}_m,p_2p_3\mathbb{Z}_m,p_2^{2}p_3\mathbb{Z}_m$ forms a $K_5$, a contradiction. Hence $\alpha_2=1$. Similarly, we find that $\alpha_3=1$. If $\beta_1\geq 2$, then $p_1\mathbb{Z}_m,p_2\mathbb{Z}_m,p_3\mathbb{Z}_m,p_1p_2\mathbb{Z}_m,p_2p_3\mathbb{Z}_m$ forms a $K_5$, a contradiction. It is easy to see that $G_{p_1}(\mathbb{Z}_{p_1^{\alpha_1}p_2p_3})$ is planar and $(4)$ holds.

\par{\bf Case 2.} $s'=2$. If $\beta_1,\beta_2\geq 2$, then $p_1\mathbb{Z}_m,p_1^{2}\mathbb{Z}_m,p_2\mathbb{Z}_m,p_1p_2\mathbb{Z}_m,p_1^{2}p_2\mathbb{Z}_m$ forms a $K_5$, a contradiction. Hence with no loss of generality we may assume that $\beta_1=1$.

\par First assume that $\beta_2\geq 2$. If $\alpha_1\geq 2$, then $p_1\mathbb{Z}_m,p_1^{2}\mathbb{Z}_m,p_2\mathbb{Z}_m,p_1p_2\mathbb{Z}_m,p_1^{2}p_2\mathbb{Z}_m$ forms a $K_5$, a contradiction. Therefore $\alpha_1=1$. If $s\geq 3$, then $p_1\mathbb{Z}_m,p_2\mathbb{Z}_m,p_1p_2\mathbb{Z}_m,p_3\mathbb{Z}_m,p_1p_3\mathbb{Z}_m$ forms a $K_5$, a contradiction. Thus $s=2$. If $\beta_2\geq 3$, then $p_1\mathbb{Z}_m,p_2\mathbb{Z}_m,p_2^{2}\mathbb{Z}_m,p_1p_2\mathbb{Z}_m,p_1p_2^{2}\mathbb{Z}_m$ forms a $K_5$, a contradiction. Therefore $\beta_2=2$. If $\alpha_2\geq 5$, then $p_2\mathbb{Z}_m,p_2^{2}\mathbb{Z}_m,p_2^{3}\mathbb{Z}_m,p_2^{4}\mathbb{Z}_m,$ $p_2^{5}\mathbb{Z}_m$ forms a $K_5$, a contradiction. Therefore $\alpha_2\leq 4$. It is easy to see that $G_{p_1p_2^{2}}(\mathbb{Z}_{p_1p_2^{\alpha_2}})$ is planar, where $\alpha_2=2,3,4$ and $(5)$ holds.

\par Now, assume that $\beta_2=1$. Let $s\geq 3$. If $\alpha_3\geq 2$, then $p_1\mathbb{Z}_m,p_3\mathbb{Z}_m,p_3^{2}\mathbb{Z}_m,p_1p_3\mathbb{Z}_m,p_1p_3^{2}\mathbb{Z}_m$ forms a $K_5$, a contradiction. Hence $\alpha_3=1$. If $\alpha_1\geq 2$, then $p_1\mathbb{Z}_m,p_1^{2}\mathbb{Z}_m,p_3\mathbb{Z}_m,p_1p_3\mathbb{Z}_m,$ $p_1^{2}p_3\mathbb{Z}_m$ forms a $K_5$, a contradiction. Therefore $\alpha_1=1$ and similarly, $\alpha_2=1$. It is clear that $G_{p_1p_2}(\mathbb{Z}_{p_1p_2p_3})$ is planar and $(6)$ holds. Now, assume that $s=2$. Then we find that $\alpha_1,\alpha_2\leq 4$. Also, one can easily check that $G_{p_1p_2}(\mathbb{Z}_{p_1^{\alpha_1}p_2^{\alpha_2}})$ is planar, where $\alpha_1,\alpha_2\leq 4$. Therefore $(7)$ holds.

\par{\bf Case 3.} $s'=3$. If $s\geq 4$, then $p_1\mathbb{Z}_m,p_2\mathbb{Z}_m,p_3\mathbb{Z}_m,p_4\mathbb{Z}_m,p_1p_4\mathbb{Z}_m$ forms a $K_5$, a contradiction. Therefore $s=3$. If $\alpha_1\geq 3$, then $p_1\mathbb{Z}_m,p_1^{2}\mathbb{Z}_m,p_1^{3}\mathbb{Z}_m,p_2\mathbb{Z}_m,p_3\mathbb{Z}_m$ forms a $K_5$, a contradiction. Therefore $\alpha_1\leq 2$ and similarly $\alpha_2,\alpha_3\leq 2$. If $\beta _1\geq 2$, then $p_1\mathbb{Z}_m,p_2\mathbb{Z}_m,p_3\mathbb{Z}_m,p_1p_2\mathbb{Z}_m,p_1p_3\mathbb{Z}_m$ forms a $K_5$, a contradiction. Therefore $\beta _1=1$ and similarly we have $\beta _2=\beta _3=1$. If $\alpha_1=\alpha_2 =2$, then $p_1\mathbb{Z}_m,p_1^{2}\mathbb{Z}_m,p_2\mathbb{Z}_m,p_2^{2}\mathbb{Z}_m,p_1p_2\mathbb{Z}_m$ forms a $K_5$, a contradiction. Therefore $\alpha_1=1$ or $\alpha_2=1$. With no loss of generality, we may assume that $\alpha_1=1$. By The same argument as we saw we find that $\alpha_2=1$ or $\alpha_3=1$. Thus we have $\alpha_1=\alpha_2=\alpha_3=1$ or $\alpha_1=\alpha_2=1$ and $\alpha_3=2$. It is clear that if $\alpha_1=\alpha_2=\alpha_3=1$, then $G_n(\mathbb{Z}_m)$ is planar. Therefore $(8)$ holds. Also, $G_{p_1p_2p_3}(\mathbb{Z}_{p_1p_2p_3^{2}})$ is not planar because $p_3\mathbb{Z}_m,p_3^{2}\mathbb{Z}_m,p_2p_3\mathbb{Z}_m,p_2p_3^{2}\mathbb{Z}_m,p_2\mathbb{Z}_m$ forms a $K_5$ which is impossible. }
\end{proof}
\begin{cor}{ Let $m$ be a positive integer number. Then $G(\mathbb{Z}_m)$ is planar if and only if $m\in \{p_1,p_1^{2},p_1^{3},p_1^{4},p_1^{5},p_1p_2,p_1p_2^{2},p_1p_2p_3\}$.}
\end{cor}
Let $G$ be a graph with $n$ vertices and $m$ edges. We recall that a {\it chord} is any edge
of $G$ joining two nonadjacent vertices in a cycle of $G$. Let $C$ be a cycle of $G$. We
say $C$ is a {\it primitive cycle} if it has no chords. Also, a graph $G$ has the {\it primitive}
cycle property (PCP) if any two primitive cycles intersect in at most one edge.
The number $frank(G)$ is called the {\it free rank} of $G$ and it is the number of primitive
cycles of $G$. Also, the number $rank(G) = m- n + r$ is called the {\it cycle rank} of $G$,
where $r$ is the number of connected components of $G$. The cycle rank of G can be
expressed as the dimension of the cycle space of $G$. By \cite[Proposition 2.2]{Gitler}, we
have $rank(G)\leq frank(G)$. A graph $G$ is called a {\it ring graph} if it satisfies in one of
the following equivalent conditions (see \cite[Theorem 2.13]{Gitler}).
\par $(i)$ $rank(G) = frank(G)$,
\par $(ii)$ $G$ satisfies the PCP and $G$ does not contain a subdivision of $K_4$ as a subgraph.
Also, an undirected graph is an {\it outerplanar} graph if it can be drawn in the plane
without crossings in such a way that all of the vertices belong to the unbounded face
of the drawing. There is a characterization for outerplanar graphs that says a graph
is outerplanar if and only if it does not contain a subdivision of $K_4$ or $K_{2,3}$. By \cite[Proposition 2.17]{Gitler}, we find that
every outerplanar graph is a ring graph and every ring graph is a planar graph.

\par\noindent{\bf Example 1.}
In Fig.1, let $v_1=p_2\mathbb{Z}_{p_1p_2^{2}},v_2=p_1p_2\mathbb{Z}_{p_1p_2^{2}},v_3=p_2^{2}\mathbb{Z}_{p_1p_2^{2}},v_4=p_1\mathbb{Z}_{p_1p_2^{2}}$. Also, in Fig.2, assume that $v_1=p_1p_2\mathbb{Z}_{p_1p_2^{3}},v_2=p_1p_2^{2}\mathbb{Z}_{p_1p_2^{3}},v_3=p_2^{2}\mathbb{Z}_{p_1p_2^{3}},v_4=p_1\mathbb{Z}_{p_1p_2^{3}},v_5=p_2\mathbb{Z}_{p_1p_2^{3}},v_6=p_2^{3}\mathbb{Z}_{p_1p_2^{3}}$
\begin{center}
    \begin{tikzpicture}
        \GraphInit[vstyle=Classic]
        \Vertex[x=1,y=0,style={black,minimum size=3pt},LabelOut=true,Lpos=270,L=$v_3$]{4}
        \Vertex[x=1,y=1,style={black,minimum size=3pt},LabelOut=true,Lpos=90,L=$v_1$]{2}
        \Vertex[x=2.3,y=0,style={black,minimum size=3pt},LabelOut=true,Lpos=270,L=$v_4$]{6}
        \Vertex[x=2.3,y=1,style={black,minimum size=3pt},LabelOut=true,Lpos=90,L=$v_2$]{3}
        \Edges(2,3)
        \Edges(6,3)
        \Edges(2,6)
        \Edges(2,4)

    \end{tikzpicture}
\hspace{4cm}
    \begin{tikzpicture}
        \GraphInit[vstyle=Classic]
        \Vertex[x=3,y=0,style={black,minimum size=3pt},LabelOut=true,Lpos=270,L=$v_6$]{5}
        \Vertex[x=0,y=0,style={black,minimum size=3pt},LabelOut=true,Lpos=270,L=$v_4$]{4}
        \Vertex[x=0,y=1,style={black,minimum size=3pt},LabelOut=true,Lpos=90,L=$v_1$]{2}
        \Vertex[x=1.5,y=0,style={black,minimum size=3pt},LabelOut=true,Lpos=270,L=$v_5$]{6}
        \Vertex[x=3,y=1,style={black,minimum size=3pt},LabelOut=true,Lpos=90,L=$v_3$]{3}
        \Vertex[x=1.5,y=1,style={black,minimum size=3pt},LabelOut=true,Lpos=90,L=$v_2$]{1}
         \Edges(3,5)
          \Edges(5,6)
           \Edges(3,6)
            \Edges(6,4)
             \Edges(2,4)
              \Edges(2,6)

    \end{tikzpicture}

 \end{center}
\hspace{1.7cm}
Fig.1 $G(\mathbb{Z}_{p_1p_2^{2}})$ \hspace{4.1cm}
Fig.2 $G_{p_1p_2^{2}}(\mathbb{Z}_{p_1p_2^{3}})$ \hspace{1.9cm}

\par\noindent{\bf Example 2.}
In Fig.3, let $v_1=p_1\mathbb{Z}_{p_1p_2p_3},v_2=p_1p_2\mathbb{Z}_{p_1p_2p_3},v_3=p_2\mathbb{Z}_{p_1p_2p_3},v_4=p_1p_3\mathbb{Z}_{p_1p_2p_3},v_5=p_3\mathbb{Z}_{p_1p_2p_3},v_6=p_2p_3\mathbb{Z}_{p_1p_2p_3}$. In Fig.4, let $v_1=p_1p_2\mathbb{Z}_{p_1p_2p_3},v_2=p_1\mathbb{Z}_{p_1p_2p_3},v_3=p_1p_3\mathbb{Z}_{p_1p_2p_3},v_4=p_2\mathbb{Z}_{p_1p_2p_3},v_5=p_3\mathbb{Z}_{p_1p_2p_3},v_6=p_2p_3\mathbb{Z}_{p_1p_2p_3}$.

\begin{center}
\begin{tikzpicture}
        \GraphInit[vstyle=Classic]
        \Vertex[x=3,y=0,style={black,minimum size=3pt},LabelOut=true,Lpos=270,L=$v_6$]{5}
        \Vertex[x=0,y=0,style={black,minimum size=3pt},LabelOut=true,Lpos=270,L=$v_4$]{4}
        \Vertex[x=0,y=1,style={black,minimum size=3pt},LabelOut=true,Lpos=90,L=$v_1$]{2}
        \Vertex[x=1.5,y=0,style={black,minimum size=3pt},LabelOut=true,Lpos=270,L=$v_5$]{6}
        \Vertex[x=1.5,y=1,style={black,minimum size=3pt},LabelOut=true,Lpos=90,L=$v_2$]{1}
        \Vertex[x=3,y=1,style={black,minimum size=3pt},LabelOut=true,Lpos=90,L=$v_3$]{3}
         \Edges(3,5)
          \Edges(5,6)
           \Edges(3,6)
            \Edges(6,4)
             \Edges(2,4)
              \Edges(2,6)

    \end{tikzpicture}
\hspace{1cm}
\begin{tikzpicture}
        \GraphInit[vstyle=Classic]
        \Vertex[x=2,y=0,style={black,minimum size=3pt},LabelOut=true,Lpos=270,L=$v_6$]{1}
        \Vertex[x=3,y=1.5,style={black,minimum size=3pt},LabelOut=true,Lpos=0,L=$v_5$]{2}
        \Vertex[x=4,y=3,style={black,minimum size=3pt},LabelOut=true,Lpos=90,L=$v_3$]{3}
        \Vertex[x=1,y=1.5,style={black,minimum size=3pt},LabelOut=true,Lpos=180,L=$v_4$]{4}
        \Vertex[x=2,y=3,style={black,minimum size=3pt},LabelOut=true,Lpos=90,L=$v_2$]{5}
        \Vertex[x=0,y=3,style={black,minimum size=3pt},LabelOut=true,Lpos=90,L=$v_1$]{6}
         \Edges(1,2)
          \Edges(2,3)
           \Edges(4,5)
            \Edges(2,5)
             \Edges(1,4)
              \Edges(4,6)
              \Edges(3,5)
          \Edges(5,6)
           \Edges(4,2)

    \end{tikzpicture}

 \end{center}
\hspace{3cm}
Fig.3 $G_{p_1p_2}(\mathbb{Z}_{p_1p_2p_3})$ \hspace{2.5cm}
Fig.4 $G(\mathbb{Z}_{p_1p_2p_3})$\hspace{4.1cm}

\par\noindent{\bf Example 3.}
In Fig.5, let $v_1=p_2\mathbb{Z}_{p_1p_2^{3}},v_2=p_2^{2}\mathbb{Z}_{p_1p_2^{3}},v_3=p_2^{3}\mathbb{Z}_{p_1p_2^{3}},v_4=p_1p_2\mathbb{Z}_{p_1p_2^{3}},v_5=p_1\mathbb{Z}_{p_1p_2^{3}},v_6=p_1p_2^{2}\mathbb{Z}_{p_1p_2^{3}}$.

\begin{center}
\begin{tikzpicture}
        \GraphInit[vstyle=Classic]
        \Vertex[x=2,y=1,style={black,minimum size=3pt},LabelOut=true,Lpos=90,L=$v_1$]{1}
        \Vertex[x=1,y=-1,style={black,minimum size=3pt},LabelOut=true,Lpos=180,L=$v_2$]{2}
        \Vertex[x=3,y=-1,style={black,minimum size=3pt},LabelOut=true,Lpos=0,L=$v_3$]{3}
        \Vertex[x=0,y=-2,style={black,minimum size=3pt},LabelOut=true,Lpos=270,L=$v_4$]{4}
        \Vertex[x=2,y=-2,style={black,minimum size=3pt},LabelOut=true,Lpos=270,L=$v_5$]{5}
        \Vertex[x=4,y=-2,style={black,minimum size=3pt},LabelOut=true,Lpos=270,L=$v_6$]{6}

         \Edges(1,2)
          \Edges(2,3)
           \Edges(3,1)


    \end{tikzpicture}
\end{center}
\hspace{5cm}
Fig.5 $G_{p_1p_2}(\mathbb{Z}_{p_1p_2^{3}})$

\begin{thm}{ Let $\mathbb{Z}_n$ be a $\mathbb{Z}_m$-module. Then $G_n(\mathbb{Z}_m)$ is a ring graph if and only if one of the following holds:
\par $(1)$ $m=p_1^{\alpha_1},n=p_1^{\beta_1}$ and $\beta_1\leq4$.
\par $(2)$ $m=p_1^{\alpha_1}p_2^{\alpha_2},n=p_1$ and $\alpha_2\leq 3$.
\par $(3)$ $m=p_1^{\alpha_1}p_2,n=p_1^{2}$.
\par $(4)$ $m={p_1^{\alpha_1}p_2p_3},n=p_1$.
\par $(5)$ $m=p_1p_2^{\alpha_2},n=p_1p_2^{2}$ and $\alpha_2=2,3$.
\par $(6)$ $m=p_1p_2p_3,n=p_1p_2$.
\par $(7)$ $m=p_1^{\alpha_1}p_2^{\alpha_2},n=p_1p_2$ and $\alpha_1,\alpha_2\leq 3$.
\par $(8)$ $m=n=p_1p_2p_3$.}
\end{thm}
\begin{proof}{ One side is obvious. For the other side assume that $G_n(\mathbb{Z}_m)$ is a ring graph. Then it is planar and by Theorem \ref{planar}, we have eight following cases:

\par{\bf Case 1.} $m=p_1^{\alpha_1},n=p_1^{\beta_1}$ and $\beta_1\leq 5$. If $\beta_1= 5$, then $p_1\mathbb{Z}_m,p_1^{2}\mathbb{Z}_m,p_1^{3}\mathbb{Z}_m,p_1^{4}\mathbb{Z}_m$ forms a $K_4$, a contradiction. Therefore $\beta_1\leq 4$. It is easy to check that $G_n(\mathbb{Z}_m)$ is a ring graph and $(1)$ holds.

\par{\bf Case 2.} $m=p_1^{\alpha_1}p_2^{\alpha_2},n=p_1$ and $\alpha_2\leq 4$. Clearly, if $\alpha_2=4$, then $p_2\mathbb{Z}_m,p_2^{2}\mathbb{Z}_m,p_2^{3}\mathbb{Z}_m,$ $p_2^{4}\mathbb{Z}_m$ forms a $K_4$, a contradiction. Hence $\alpha_2 \leq 3$. Now, $G_n(\mathbb{Z}_m)$ has at most three non-isolated vertices and so it is a ring graph and $(2)$ holds.

\par{\bf Case 3.} $m=p_1^{\alpha_1}p_2,n=p_1^{2}$. In this case, $\{p_1\mathbb{Z}_m,p_2\mathbb{Z}_m,p_1p_2\mathbb{Z}_m\}$ is the set of all non-isolated vertices of $G_n(\mathbb{Z}_m)$. This yields that $G_n(\mathbb{Z}_m)$ is a ring graph and $(3)$ holds.

\par{\bf Case 4.} $m=p_1^{\alpha_1}p_2p_3,n=p_1$. It is easy to see that $G_{p_1}(\mathbb{Z}_{p_1^{\alpha_1}p_2p_3})$ is a ring graph and $(4)$ holds.

\par{\bf Case 5.} $m=p_1p_2^{\alpha_2},n=p_1p_2^{2}$ and $\alpha_2=2,3,4$. If $\alpha_2=4$, then $p_2\mathbb{Z}_m,p_2^{2}\mathbb{Z}_m,p_2^{3}\mathbb{Z}_m,p_2^{4}\mathbb{Z}_m$ forms a $K_4$, a contradiction. Therefore $\alpha_2=2,3$. By Fig.1 and Fig.2, $G_{p_1p_2^{2}}(\mathbb{Z}_{p_1p_2^{2}})$ and $G_{p_1p_2^{2}}(\mathbb{Z}_{p_1p_2^{3}})$ are ring graphs and $(5)$ holds.

\par{\bf Case 6.} $m=p_1p_2p_3,n=p_1p_2$. In this case, by Fig.3, we conclude that $G_n(\mathbb{Z}_m)$ is a ring graph and $(6)$ holds.

\par{\bf Case 7.} $m=p_1^{\alpha_1}p_2^{\alpha_2},n=p_1p_2$ and $\alpha_1,\alpha_2\leq 4$. If $\alpha_1=4$, then $p_1\mathbb{Z}_m,p_1^{2}\mathbb{Z}_m,p_1^{3}\mathbb{Z}_m,p_1^{4}\mathbb{Z}_m$ forms a $K_4$, a contradiction. Hence $\alpha_1\leq3$ and similarly, $\alpha_2\leq3$. Now, by \cite[Theorem 11]{heyGMz}, we have  $G_n(\mathbb{Z}_m)\setminus \mathfrak{A}\cong K_{\alpha_1}\cup K_{\alpha_2}$, where $\mathfrak{A}$ is the set of all non-isolated vertices of $G_n(\mathbb{Z}_m)$. Therefore $(7)$ holds.(see Fig.5)

\par{\bf Case 8.} $m=n=p_1p_2p_3$. By Fig.4, we find that $G_n(\mathbb{Z}_m)$ is a ring graph and $(8)$ holds.}
\end{proof}

\begin{cor}{ Let $m$ be a positive integer number. Then $G(\mathbb{Z}_m)$ is a ring graph if and only if $m\in \{p_1,p_1^{2},p_1^{3},p_1^4,p_1p_2,p_1p_2^{2},p_1p_2p_3\}$.}
\end{cor}
Using theorems in this section and the proof of pervious theorem, we have the following result.
\begin{thm}{ Let $\mathbb{Z}_n$ be a $\mathbb{Z}_m$-module. Then $G_n(\mathbb{Z}_m)$ is outerplanar if and only if one of the following holds:
\par $(1)$ $m=p_1^{\alpha_1},n=p_1^{\beta_1}$ and $\beta_1\leq4$.
\par $(2)$ $m=p_1^{\alpha_1}p_2^{\alpha_2},n=p_1$ and $\alpha_2\leq 3$.
\par $(3)$ $m=p_1^{\alpha_1}p_2,n=p_1^{2}$.
\par $(4)$ $m={p_1^{\alpha_1}p_2p_3},n=p_1$.
\par $(5)$ $m=p_1p_2^{\alpha_2},n=p_1p_2^{2}$ and $\alpha_2=2,3$.
\par $(6)$ $m=p_1p_2p_3,n=p_1p_2$.
\par $(7)$ $m=p_1^{\alpha_1}p_2^{\alpha_2},n=p_1p_2$ and $\alpha_1,\alpha_2\leq 3$.
\par $(8)$ $m=n=p_1p_2p_3$.}
\end{thm}
Also, we have the following corollary.
\begin{cor}{ Let $m$ be a positive integer number. Then $G(\mathbb{Z}_m)$ is outerplanar if and only if $m\in \{p_1,p_1^2,p_1^3,p_1^4,p_1p_2,p_1p_2^{2},p_1p_2p_3\}$.}
\end{cor}

\end{document}